%% file: leb12.tex
\documentclass[a4paper]{article}

\usepackage{graphicx}
\usepackage{color}
\usepackage{stmaryrd}
\usepackage[all]{xy}

\usepackage[numbers]{natbib}
\usepackage{natbibspacing}
\renewcommand{\bibfont}{\small}

\definecolor{myurlcolor}{rgb}{0.1,0.1,0.8}
\definecolor{mylinkcolor}{rgb}{0.05,0.05,0.4}
\usepackage{hyperref}
\hypersetup{colorlinks,
linkcolor=mylinkcolor,
citecolor=mylinkcolor,
urlcolor=mylinkcolor}

\input{ab}

\input{thmlist}

\title{A categorical derivation of Lebesgue integration}
\author{Tom Leinster%
\thanks{School of Mathematics, University of Edinburgh, James Clerk Maxwell Building, Peter Guthrie
Tait Road, Edinburgh EH9 3FD, Scotland. Email:  
  Tom.Leinster@ed.ac.uk. MSC2020: 18A40, 28C99, 46B99}}
\date{}

\begin{document}

\sloppy
\maketitle

\begin{abstract}
We identify simple universal properties that uniquely characterize the
Lebesgue $L^p$ spaces. There are two main theorems. The first states that
the Banach space $L^p[0, 1]$, equipped with a small amount of extra
structure, is initial as such.  The second states that the $L^p$ functor on
finite measure spaces, again with some extra structure, is also initial as
such. In both cases, the universal characterization of the
\emph{integrable} functions produces a unique characterization of
\emph{integration}. We use the universal properties to derive some of the
basic elements of integration theory. We also state universal properties
characterizing the sequence spaces $\ell^p$ and $c_0$, as well as the
functor $L^2$ taking values in Hilbert spaces. 
\end{abstract}

\section{Introduction}
\label{sec:intro}

Lebesgue integration is universally agreed to lie at the heart of analysis,
yet its fundamental nature contrasts with the long and perhaps
technical-seeming string of preliminaries that faces the student wishing to
learn the basic definitions.  For instance, a standard route to defining
Lebesgue integrability and integration for functions on $\R$ involves
(1)~defining null sets and almost everywhere convergence, (2)~defining step
functions, (3)~defining the set $\mathcal{L}^\text{inc}$ of functions that
are almost everywhere limits of some increasing sequence of step functions,
(4)~defining the set $\mathcal{L}^1$ of differences of elements of
$\mathcal{L}^\text{inc}$, (5)~defining integration for step functions, and
(6)~proving that the definition is consistent, before finally (7)~extending
the integral to $\mathcal{L}^1$ by linearity and continuity, and
(8)~passing from $\mathcal{L}^1$ to its quotient $L^1$. A different route
constructs $L^1$ as the completion of the normed space of continuous
functions; but for that, one must first develop the theory of integration
in the continuous case.

One might therefore wish for a description of Lebesgue integration that is
as simple and direct as the concept is fundamental. This paper presents two such
theorems. One uniquely characterizes the space $L^1[0, 1]$ and the operator
$\int_0^1$. The other does the same on an arbitrary measure space. Both
theorems entirely bypass steps (1)--(8). 

The theorems characterize spaces of integrable functions uniquely
up to isomorphism via a universal property. Specifically, they characterize
them as initial objects in certain categories.

Let us recall this notion. An object $Z$ of a category $\cat{C}$ is
\demph{initial} if for each object $C$ of $\cat{C}$, there is exactly one
map $Z \to C$ in $\cat{C}$. For example, the empty set is initial in the
category of sets, the trivial group is initial in the category of groups,
and $\Z$ is initial in the category of rings with identity. A slightly less
trivial example, closer in shape to those we will consider, is as
follows. There is a category whose objects are triples $(X, x, r)$
consisting of a set $X$, an element $x \in X$ and a function $r \from X \to
X$, and whose maps are functions preserving this structure. Its initial
object is $(\nat, 0, s)$, where $s$ is the successor function $n \mapsto n
+ 1$. Concretely, initiality means that for any set $X$, element $x$ and
function $r \from X \to X$, there is a unique sequence $(x_n)_{n \geq 0}$
in $X$ satisfying $x_0 = x$ and $x_{n + 1} = r(x_n)$.

Any two initial objects of a category are isomorphic. Hence, any theorem
stating that an object is initial in some category characterizes it
uniquely up to isomorphism. Our first theorem does this for the Banach
space $L^p[0, 1]$ equipped with a small amount of extra structure. Our
second theorem does it for the functor $L^p$ from finite measure spaces to
Banach spaces, again with some extra structure. (A \demph{finite measure
space} is a measure space with finite total measure; the underlying set
need not be finite.) Thus, both characterize $L^p$ uniquely up to isometric
isomorphism. And in both cases, the initiality of $L^1$ leads swiftly to a
unique characterization of the integration operator.

What is the point of these theorems?

First, they enable us to leapfrog all the customary preliminary
definitions, directly characterizing Lebesgue integrability and Lebesgue
integration. 

Second, any theorem stating that some object is initial establishes
uniqueness at two levels: the uniqueness up to isomorphism of the object
itself, and the literal uniqueness of the map to any other object.  Such a
theorem is effectively a large family of uniqueness theorems, one for each
object of the category.  Generally, for any important mathematical object,
one can ask: is it the only object enjoying the fundamental properties that
it enjoys? If not, why do we use it rather than something else? Or if so,
can we prove it?  For example, theorems of Alesker, Artstein-Avidan and
Milman answer these questions for the Fourier and Legendre transforms
\cite{AAAM,AAM}, and the present work answers them for the $L^p$ spaces and
Lebesgue integration. Theorem~\ref{thm:Ap}, for instance, characterizes the
space $L^p[0, 1]$ uniquely up to isometric isomorphism, and
Corollary~\ref{cor:elem-integral-intvl} characterizes the integration
operator on $[0, 1]$ uniquely. It is an entirely elementary
characterization of integration, using no categorical terminology. 

Third, the main theorems clarify conceptual dependencies. They show that,
granted some general categorical language, the concepts of Lebesgue
integrability and integration arise inevitably from little more than the
concept of Banach space. Perhaps surprisingly, they arise
\emph{automatically}, without invoking any prior concept of area under the
curve or antidifferentiation.

Fourth, the second main theorem, characterizing the $L^p$ functors,
provides a guide for the discovery of new theories of integration. A
researcher seeking the right notion of integration in some new context
(perhaps some new kind of function on a new kind of space) could follow the
same template: decide what kind of spaces the integrable functions should
form and what kind of functoriality should hold, formulate a universal
property analogous to the one below, and find the functor satisfying it.

Finally, such theorems are important simply because the Lebesgue theory is
important. There is no question of displacing the twin classical
perspectives on integration, area under the curve and the inverse of
differentiation. But the more perspectives the better, and here we provide
a new one: its characterization by a universal property.

The deep theorems on integration and measure are particular to that
subject, so there is no chance of deriving them using general categorical
methods that would apply equally to other categories.  The fact that (for
instance) $L^1[0, 1]$ is the initial object of a certain category does
determine it uniquely, so all of its properties could in principle be
derived from that characterization, but the deeper properties would
inevitably use very specific features of the category in which it is
initial. We prove no deep theorems of analysis here. However, our
results do establish that the fundamental objects of Lebesgue's theory can be
characterized in extremely primitive terms.

\begin{figure}
\setlength{\unitlength}{1mm}
\setlength{\fboxsep}{0pt}
\begin{picture}(120,18)(0,2)
\cell{60}{21}{t}{\includegraphics[width=60mm]{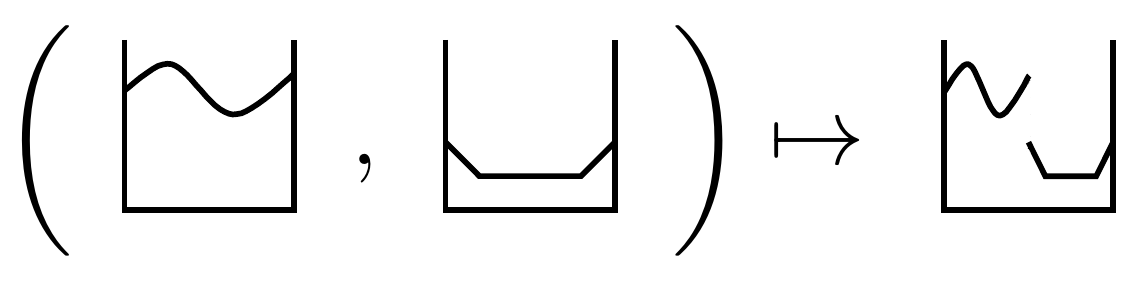}}
\cell{36.5}{9}{t}{$\scriptstyle{0}$}
\cell{45.5}{9}{t}{$\scriptstyle{1}$}
\cell{41}{4}{c}{$f$}
\cell{53.5}{9}{t}{$\scriptstyle{0}$}
\cell{62.5}{9}{t}{$\scriptstyle{1}$}
\cell{58}{4}{c}{$g$}
\cell{79.5}{9}{t}{$\scriptstyle{0}$}
\cell{88.5}{9}{t}{$\scriptstyle{1}$}
\cell{84}{4}{c}{$\gamma(f, g)$}
\end{picture}
\caption{The map $\gamma \from L^1[0, 1] \oplus L^1[0, 1] \to L^1[0,
1]$.}
\label{fig:juxt}
\end{figure}

We now summarize the results.

Section~\ref{sec:A} concerns integration on
$[0, 1]$. The most important case of the main theorem is as follows.
Let $\cat{A}$ be the category in which an object $(V, v, \delta)$ is a
Banach space $V$ together with a point $v$ in its closed unit ball and a
linear contraction $\delta \from V \oplus V \to V$ such that $\delta(v, v)
= v$. Here we give $V \oplus V$ the norm $\|(v_1, v_2)\| = \hlf(\|v_1\| +
\|v_2\|)$.

One object of $\cat{A}$ is $L^1[0, 1]$ together with the constant function
$1$ and the map $\gamma \from L^1[0, 1] \oplus L^1[0, 1] \to L^1[0, 1]$
that juxtaposes two functions and scales the domain by a factor of $1/2$
(Figure~\ref{fig:juxt}). We prove:

\begin{thma}
The initial object of $\cat{A}$ is $(L^1[0, 1], I, \gamma)$.
\end{thma}

Thus, for each object $(V, v, \delta)$ of $\cat{A}$, there is a unique map
$(L^1[0, 1], I, \gamma) \to (V, v, \delta)$ in $\cat{A}$.  Certain choices
of $(V, v, \delta)$ are especially consequential. For example, there is an
object of $\cat{A}$ consisting of the ground field $\F$ (either $\R$ or
$\C$), together with $1 \in \F$ and the arithmetic mean $m\from \F \oplus
\F \to \F$. We prove that the unique map $ (L^1[0, 1], I, \gamma) \to (\F,
1, m) $ in $\cat{A}$ is the integration operator $\int_0^1$.

The universal characterization of the \emph{integrable} functions therefore
gives rise to a unique characterization of \emph{integration}.

In fact, Theorem~A is just the case $p = 1$ of a general result
characterizing $(L^p[0, 1], I, \gamma)$ as the initial object of a category
$\cat{A}^p$ (Theorem~\ref{thm:Ap}). Here $1 \leq p < \infty$. When $p =
\infty$, the analogous result characterizes not $L^\infty[0, 1]$ but the
space $C(\{0, 1\}^\nat)$ of continuous functions on the Cantor set
(Proposition~\ref{propn:Ainfty}).

The key to Theorem~A is that integration on $[0, 1]$ is a continuous notion
of mean, and means are built into the category $\cat{A} = \cat{A}^1$ via
the definition of the norm on $V \oplus W$. To define $\cat{A}^p$ for $p >
1$, we simply replace the arithmetic mean in the definition of $\cat{A}$ by
the power mean of order $p$. 

Abstract characterizations are all well and good, but it is natural to want
to realize $L^p[0, 1]$ as a \emph{function} space. We do so in two
senses. First, by an observation of Meckes (Proposition~\ref{propn:mm}),
the universal property of $L^1[0, 1]$ produces a canonical map $L^1[0, 1]
\to C[0, 1]$, which in concrete terms maps $f \in L^1[0, 1]$ to the
continuous function $x \mapsto \int_0^x f$. Differentiating $\int_0^x f$,
we recover $f$ as a function (up to equality almost everywhere). In the
opposite direction, the universal property of $C(\{0, 1\}^\nat)$ leads to a
canonical map $C[0, 1] \to L^p[0, 1]$ for every $p$, which in concrete
terms is the usual inclusion (Proposition~\ref{propn:A-incl-infty}).

Repeatedly exploiting the universal property of $L^p[0, 1]$, we derive
further maps fundamental in analysis, including the canonical pairing
\[
L^p[0, 1] \times L^q[0, 1] \to \F
\]
for conjugate exponents $p$ and $q$ (Proposition~\ref{propn:conj-prod}).

The universal property of $L^p[0, 1]$ stems from the self-similarity of
$[0, 1]$ (see `Related work' below), and the same idea can be used to
characterize $L^p(X)$ for other self-similar spaces
$X$. Proposition~\ref{propn:seq-spaces} is a universal characterization of
the sequence space $\ell^p$ (the case $X = \nat$) for each $p \in [1,
\infty)$, and also of the sequence space $c_0$.

The results described so far uniquely \emph{characterize} the spaces
concerned, but general categorical techniques actually \emph{construct}
them (Remark~\ref{rmk:alg-endo}). Specifically, $L^p[0, 1]$ is the initial
algebra for a certain endofunctor on the category of pointed Banach spaces,
and a standard theorem of Ad\'amek on initial algebras constructs it for
us. The construction realizes $L^p[0, 1]$ as the colimit, in the category
of Banach spaces, of the spaces $E_n$ of step functions whose points of
discontinuity are integer multiples of $2^{-n}$. Similar constructions
produce $C(\{0, 1\}^\nat)$, $\ell^p$ and $c_0$. The moral is that even if
Lebesgue had never formulated his theory, categorical machinery would still
construct its basic objects.

Section~\ref{sec:B} concerns functions on an arbitrary finite measure
space.  The $L^1$ construction on measure spaces is functorial in two ways:
contravariantly with respect to measure-preserving maps and covariantly
with respect to embeddings. To combine the two, let $\Mpar$ be the category
of finite measure spaces and measure-preserving \emph{partial} maps.
Then $L^1$ defines a functor $\Mpar^\op \to \Ban$.

Our second main theorem characterizes this functor $L^1$ together with, for
each $X$, the function $I_X \in L^1(X)$ with constant value $1$.  We define
a category $\cat{B}$ of pairs $(F, v)$ consisting of a functor $\Mpar^\op
\to \Ban$ and an element $v_X \in F(X)$ for each $X$, subject to some
simple axioms. Then:

\begin{thmb}
The initial object of $\cat{B}$ is $(L^1, I)$.
\end{thmb}

As for Theorem~A, the universal characterization of the integrable
functions produces a unique characterization of integration. Indeed, there
is another object $(\F, t)$ of $\cat{B}$, where $\F$ is the constant
functor whose value is the ground field and $t_X$ is the total measure of a
measure space $X$. The unique map $(L^1, I) \to (\F, t)$ in $\cat{B}$ is
nothing but integration.

A variant of Theorem~B uses the functoriality of $L^1$ with respect to
embeddings only. Thus, $\Mpar^\op$ is replaced by the category of measure
spaces and embeddings, and $\cat{B}$ is replaced by a simpler category
$\Bemb$. We prove that $(L^1, I)$ is also initial in $\Bemb$. As a
consequence, we obtain the construction $(f, \mu) \mapsto f \dee\mu$
(Proposition~\ref{propn:B-action}). This in turn allows us, via the
Radon--Nikodym theorem, to realize elements of the abstractly characterized
space $L^1(X)$ as equivalence classes of functions on $X$.

At the cost of a further axiom, Theorem~B can be generalized to any $p \in
[1, \infty]$. That is, $(L^p, I)$ is the initial object of a suitably
defined category $\Bpar^p$ (and, similarly, of $\Bemb^p$), where $\Bpar^1 =
\cat{B}$. This is Theorem~\ref{thm:Bp}. In the case $p = 2$, the functor
$L^2$ takes values in the category of Hilbert spaces, and also possesses a
slightly simpler universal property (Proposition~\ref{propn:Hilb}).

Although Theorem~B concerns arbitrary finite measure spaces and Theorem~A
concerns only the space $[0, 1]$, Theorem~A is not a special case of
Theorem~B. Indeed, the hypotheses of Theorem~A make no mention of the set
$[0, 1]$, let alone its $\sigma$-algebra or Lebesgue measure on it, whereas
the starting point for Theorem~B is the category of measure spaces. To my
knowledge, there is no reasonable way to derive Theorem~A from Theorem~B.

\paragraph{Related work}
This work arose from a universal characterization of the real interval by
Freyd~\cite{FreARA} (itself related to a characterization by Escard\'o and
Simpson~\cite{EsSi}). A topological variant of Freyd's result, due to the
author, runs as follows \cite[Theorem 2.5]{GSSO}. A \demph{bipointed space}
is a topological space with an ordered pair of distinct, closed basepoints;
an example is $[0, 1]$ with $0$ and $1$. Two bipointed spaces $X$ and $Y$
can be joined to form a new bipointed space $X \vee Y$, identifying the
second basepoint of $X$ with the first of $Y$; for example, $[0, 1] \vee
[0, 1] \cong [0, 2]$. The theorem is that in the category of bipointed
spaces $X$ equipped with a basepoint-preserving continuous map $X \to X
\vee X$, the terminal object is $[0, 1]$ (topologized as usual) with the
map $2 \times \dashbk\from [0, 1] \to [0, 2]$. (This goes some way towards
explaining the special role of $[0, 1]$ in homotopy theory.) Theorem~A is a
kind of dual to this, using the same self-similarity of the interval to
exhibit $L^1[0, 1]$ as an \emph{initial} object.

Algebraic approaches to integration go at least as far back as the 1965
work of Irving Segal~\cite{SegaAIT}. For example, he proved (p.~432) that
every commutative $\R$-algebra equipped with a linear functional $\int$
satisfying certain axioms must be a dense subalgebra of $L^\infty(X)$ for
some measure space $X$.

\paragraph{Convention} 
We work throughout over a field $\F$, which is either $\R$ or $\C$.

\section{Integration on $[0, 1]$}
\label{sec:A}

Let $p \in [1, \infty)$. In this section, we uniquely characterize the
Banach space $L^p[0, 1]$ together with two further pieces of data: the
function $I \in L^p[0, 1]$ taking constant value $1$, and the juxtaposition
map
\[
\gamma \from L^p[0, 1] \oplus L^p[0, 1] \to L^p[0, 1]
\]
defined on $f, g \in L^p[0, 1]$ by
\[
(\gamma(f, g))(x)
=
\begin{cases}
f(2x)           &\text{if } x < 1/2,    \\
g(2x - 1)       &\text{if } x > 1/2
\end{cases}
\]
(Figure~\ref{fig:juxt}). We also show how categorical methods not only
characterize $L^p[0, 1]$ uniquely, but also construct it explicitly
(Remark~\ref{rmk:alg-endo}).

For us, a \demph{map} of Banach spaces is a linear contraction (map with
norm less than or equal to $1$), and $\Ban$ is the category of Banach
spaces and maps between them. Thus, the isomorphisms in $\Ban$ are the
\emph{isometric} isomorphisms. For Banach spaces $V$ and $W$, let $V
\oplus_p W$ denote the direct sum with norm
\[
\|(v, w)\|_p = \Bigl( \hlf \bigl( \|v\|^p + \|w\|^p \bigr) \Bigr)^{1/p}.
\]
Let $\cat{A}^p$ be the category whose objects are triples $(V, v, \delta)$,
where $V$ is a Banach space, $v \in V$ with $\|v\| \leq 1$, and
$\delta\from V \oplus_p V \to V$ is a map of Banach spaces satisfying
$\delta(v, v) = v$.  The maps $(V', v', \delta') \to (V, v, \delta)$ in
$\cat{A}^p$ are the maps $\theta\from V' \to V$ in $\Ban$ that preserve the
structure: 
\[
\theta(v') = v,
\qquad
\theta(\delta'(v'_1, v'_2)) = \delta(\theta(v'_1), \theta(v'_2))
\]
for all $v'_1, v'_2 \in V'$. For example, the category $\cat{A}$ of the
Introduction is $\cat{A}^1$.

The map $\gamma$ is an isometric isomorphism $L^p[0, 1] \oplus_p L^p[0, 1]
\to L^p[0, 1]$, and in particular, a contraction. Hence $(L^p[0, 1], I,
\gamma)$ is an object of $\cat{A}^p$.

\begin{thm}[Universal property of \textmd{$L^p[0, 1]$}]
\label{thm:Ap}
Let $1 \leq p < \infty$. Then
$(L^p[0, 1], I, \gamma)$ is the initial object of $\cat{A}^p$.
\end{thm}

\begin{proof}
For $n \geq 0$, let $\dyfn_n$ be the subspace of $L^p[0, 1]$ consisting of
the equivalence classes of step functions constant on each of the intervals
$((i - 1)/2^n, i/2^n)$ ($1 \leq i \leq 2^n$). Write $\dyfn = \bigcup_{n
\geq 0} \dyfn_n$, which is the space of step functions whose points of
discontinuity are dyadic rationals.

The assumption that $p < \infty$ implies that $\dyfn$ is dense in the set
of all step functions on $[0, 1]$, which in turn is dense in $L^p[0, 1]$;
so $\dyfn$ is dense in $L^p[0, 1]$.  It follows that $L^p[0, 1]$ is the
colimit (direct limit) of the diagram $\dyfn_0 \incl \dyfn_1 \incl \cdots$
in $\Ban$ \cite[Examples~2.2.4(h) and~2.2.6(g)]{BorcHCA1}.  Also note that
$\gamma$ restricts to an isomorphism $\dyfn_n \oplus_p \dyfn_n \to \dyfn_{n
+ 1}$ for each $n \geq 0$.

Now let $(V, v, \delta) \in \cat{A}^p$.  We must show that there exists a
unique map $(L^p[0, 1], I, \gamma) \to (V, v, \delta)$ in $\cat{A}^p$.

\paragraph{Uniqueness}
Let $\theta$ be a map $(L^p[0, 1], I, \gamma) \to (V, v, \delta)$ in
$\cat{A}^p$. Then $\theta(I) = v$, which by linearity determines
$\theta|_{\dyfn_0}$ uniquely. Suppose inductively that $\theta|_{\dyfn_n}$
is determined uniquely. Since $\theta$ is a map in $\cat{A}^p$, the square
\[
\xymatrix{
\dyfn_n \oplus_p \dyfn_n 
\ar[r]^{\gamma}
\ar[d]_{\theta|_{\dyfn_n} \oplus \theta|_{\dyfn_n}}     &
\dyfn_{n + 1}
\ar[d]^{\theta|_{\dyfn_{n + 1}}}  \\
V \oplus_p V   
\ar[r]_\delta   &
V
}
\]
commutes. But $\gamma \from \dyfn_n \oplus_p \dyfn_n \to \dyfn_{n + 1}$ is
invertible, so $\theta|_{\dyfn_{n + 1}}$ is uniquely determined by
$\theta|_{\dyfn_n}$, completing the induction. Hence $\theta$ is uniquely
determined on the dense subspace $\dyfn$ of $L^p[0, 1]$, and so, as
$\theta$ is bounded, on $L^p[0, 1]$ itself.

\paragraph*{Existence}
For each $n \geq 0$, define a map $\theta_n \from \dyfn_n \to V$ in
$\Ban$ as follows: $\theta_0 \from \dyfn_0 \iso \F \to V$ is given by
$\theta_0(I) = v$ (and is a contraction because $\|v\| \leq 1$), and
inductively,
\[
\theta_{n + 1}
=
\Bigl(
\dyfn_{n + 1} \ltoby{\gamma^{-1}}
\dyfn_n \oplus_p \dyfn_n \ltoby{\theta_n \oplus \theta_n}
V \oplus_p V \toby{\delta}
V
\Bigr).
\]
Using the axiom that $\delta(v, v) = v$, one checks that $\theta_{n + 1}$
extends $\theta_n$ for each $n \geq 0$. Since $L^p[0, 1]$ is the colimit of
$\dyfn_0 \incl \dyfn_1 \incl \cdots$, there is a unique map $\theta
\from L^p[0, 1] \to V$ such that $\theta|_{\dyfn_n} = \theta_n$ for each
$n$.

It remains to prove that $\theta$ is a map $(L^p[0, 1], I, \gamma) \to (V,
v, \delta)$ in $\cat{A}^p$.  First, $\theta(I) = \theta_0(I) = v$.  Second,
we must show that the lower square of the diagram
\[
\xymatrix@M+.3mm{
\dyfn_n \oplus_p \dyfn_n 
\ar@/_4pc/[dd]_{\theta_n \oplus \theta_n} 
\ar@{^{(}->}[d]
\ar[r]^-{\gamma}       
&
\dyfn_{n + 1} 
\ar@/^4pc/[dd]^{\theta_{n + 1}}   
\ar@{^{(}->}[d]                 
\\
L^p[0, 1] \oplus_p L^p[0, 1]      
\ar[d]_{\theta \oplus \theta}   
\ar[r]^-{\gamma}        
&
L^p[0, 1]
\ar[d]^\theta   
\\
V \oplus_p V 
\ar[r]_-{\delta}        
&
V
}
\]
commutes.  The upper square commutes trivially, the outer square commutes
by definition of $\theta_{n + 1}$, and the triangles commute by definition
of $\theta$.  Thus, the lower square commutes on the subspace $\dyfn_n
\oplus_p \dyfn_n$ of $L^p[0, 1] \oplus L^p[0, 1]$, for each $n$.  But
\[
\bigcup_{n \geq 0} \dyfn_n \oplus_p \dyfn_n 
=
\dyfn \oplus_p \dyfn,
\]
and $E \oplus_p E$ is dense in $L^p[0, 1] \oplus_p L^p[0, 1]$ because $E$
is dense in $L^p[0, 1]$ and $\oplus_p$ induces the product topology. Hence
the lower square does commute.
\end{proof}

Taking $p = 1$, we immediately obtain a characterization of integration.
Let $m: \F \oplus_1 \F \to \F$ denote the arithmetic mean: $m(x, y) =
\hlf(x + y)$. Then $(\F, 1, m)$ is an object of $\cat{A}^1$.

\begin{propn}[Uniqueness of integration]
\label{propn:A-integration}
The unique map
\[
(L^1[0, 1], I, \gamma) \to (\F, 1, m)
\]
in $\cat{A}^1$ is the integration operator $\int_0^1$.
\end{propn}

\begin{proof}
By Theorem~\ref{thm:Ap}, it suffices to prove that $\int_0^1$ is a map in
$\cat{A}^1$. This statement amounts to the linearity of $\int_0^1$ together
with the properties
\begin{align*}
\left| \int_0^1 f \right|       &
\leq 
\int_0^1 |f|,   \\
\int_0^1 I      &
=
1,      \\              
\int_0^1 \gamma(f, g)   &
=
\frac{1}{2} \biggl( \int_0^1 f + \int_0^1 g \biggr)
\end{align*}
($f, g \in L^1[0, 1]$).
\end{proof}

There is an entirely elementary corollary, using no categorical language:

\begin{cor}
\label{cor:elem-integral-intvl}
$\int_0^1$ is the unique bounded linear functional on $L^1[0, 1]$
such that $\int_0^1 1 = 1$ and 
\[
\int_0^1 f(x) \dx	
=
\frac{1}{2}
\biggl(
\int_0^1 f \Bigl(\frac{x}{2}\Bigr) \dx
+
\int_0^1 f \Bigl(\frac{x + 1}{2}\Bigr) \dx
\biggr)
\]
for all $f \in L^1[0, 1]$.
\end{cor}

\begin{proof}
This is just a restatement of Proposition~\ref{propn:A-integration}, except
that the hypothesis that $\int_0^1$ is a contraction has been weakened to
boundedness. That suffices because the uniqueness part of the proof of
Theorem~\ref{thm:Ap} uses only boundedness, not contractivity, of the maps
in $\cat{A}^p$.  
\end{proof}

The universal property of $L^1[0, 1]$ also produces integration on
subintervals of $[0, 1]$. The following result is due to Mark Meckes
(personal communication).

Write $C_*[0, 1]$ for the Banach space of continuous functions $F \from [0,
1] \to \F$ such that $F(0) = 0$, with the sup norm. Define $i \in C_*[0,
1]$ by $i(x) = x$, and
\[
\kappa \from C_*[0, 1] \oplus C_*[0, 1] \to C_*[0, 1]
\]
by
\[
(\kappa(F, G))(x)
=
\begin{cases}
\hlf F(2x)              &
\text{if } 0 \leq x \leq \hlf, \\
\hlf \bigl(F(1) + G(2x - 1)\bigr)       &
\text{if } \hlf \leq x \leq 1
\end{cases}
\]
($F, G \in C_*[0, 1]$). Then $(C_*[0, 1], i, \kappa)$ is an object of
$\cat{A}^1$. 

\begin{propn}[Meckes]
\label{propn:mm}
The unique map $(L^1[0, 1], I, \gamma) \to (C_*[0, 1], i, \kappa)$ in
$\cat{A}^1$ is the definite integration operator 
\[
\begin{array}{cccc}
\int_0^\dashbk \from    &L^1[0, 1]      &\to    &C_*[0, 1]      \\[1ex]
                        &f              &\mapsto&\int_0^x f.
\end{array}
\]
\end{propn}

\begin{proof}
By Theorem~\ref{thm:Ap}, it suffices to show that $\int_0^\dashbk$ is a map
in $\cat{A}^1$. This amounts to the linearity of integration together with
the elementary facts that
\begin{align*}
\left| \int_0^x f \right|   &
\leq
\int_0^1 |f|,   \\
\int_0^x 1      &
=
x,      \\
\int_0^x \gamma(f, g)   &
=
\biggl( 
\kappa \biggl( \int_0^\dashbk f, \int_0^\dashbk g \biggr) 
\biggr) (x)
\end{align*}
for all $f, g \in L^1[0, 1]$ and $x \in [0, 1]$.
\end{proof}

Theorem~\ref{thm:Ap} uniquely characterizes $L^1[0, 1]$ as an
\emph{abstract} Banach space, but Proposition~\ref{propn:mm} allows us to
realize its elements as equivalence classes of \emph{functions}. Given an
element $\alpha \in L^1[0, 1]$, first apply the map of
Proposition~\ref{propn:mm} to obtain an element of $C_*[0, 1]$, then
differentiate to obtain a function defined almost everywhere on $[0,
1]$. This function is a representative of $\alpha$, since every integrable
function $f$ satisfies $f(x) = \frac{d}{dx} \int^x_0 f$ for almost all $x$.

\begin{remarks}
\label{rmks:cantor}
\begin{enumerate}
\item
The analogue of Theorem~\ref{thm:Ap} for $p = \infty$ is false. Let
$\oplus_\infty$ denote the direct sum with norm $\|(v, w)\| = \max\{\|v\|,
\|w\|\}$. Define $\cat{A}^\infty$ analogously to $\cat{A}^p$. Then by the
same argument as for $p < \infty$, the initial object of $\cat{A}^\infty$
is the closure of $\dyfn$ in $L^\infty[0, 1]$, together with $I$ and
$\gamma$. This is not $L^\infty[0, 1]$; for example, $\overline{\dyfn}$
does not contain the indicator function $I_{[0, 1/3]}$.

\item
\label{rmk:cantor-replace}
In Theorem~\ref{thm:Ap}, $[0, 1]$ can equivalently be replaced by 
Cantor space $\{0, 1\}^\nat$ with the probability measure in
which the set of sequences beginning with $n$ prescribed bits has measure
$2^{-n}$. The measure-preserving surjection
\[
\begin{array}{cccc}
s \from &\{0, 1\}^\nat          &\to            &[0, 1] \\
        &(x_0, x_1, \ldots)     &\mapsto        &
\displaystyle\sum_{n = 0}^\infty x_n 2^{-(n + 1)}
\end{array}
\]
induces an isomorphism $L^p[0, 1] \iso L^p( \{0, 1\}^\nat )$ for
each $p \in [1, \infty]$.  Under this isomorphism, $\gamma$ corresponds to
the map
\[
\gamma:  
L^p\bigl( \{0, 1\}^\nat \bigr) \oplus L^p\bigl( \{0, 1\}^\nat \bigr)
\to
L^p\bigl( \{0, 1\}^\nat \bigr)
\]
defined by
\begin{equation}
\label{eq:cantor-gamma-def}
\bigl(\gamma(f, g)\bigr)(x_0, x_1, \ldots)
=
\begin{cases}
f(x_1, x_2, \ldots)     &\text{if } x_0 = 0,    \\
g(x_1, x_2, \ldots)     &\text{if } x_0 = 1
\end{cases}
\end{equation}
($f, g \in L^p(\{0, 1\}^\nat)$, $x_i \in \{0, 1\}$). Thus,
$\bigl( L^p( \{0, 1\}^\nat), I, \gamma \bigr)$ is initial in
$\cat{A}^p$ whenever $p < \infty$, where $I$ is the constant function $1$
on $\{0, 1\}^\nat$.
\end{enumerate}
\end{remarks}

We now show that the initial object of $\cat{A}^\infty$ consists of not
\emph{bounded} functions, but \emph{continuous} functions, on the Cantor
space rather than the interval.  Give $\{0, 1\}^\nat$ the product topology
and $C(\{0, 1\}^\nat)$ the sup norm. The map
\[
\gamma \from 
C\bigl(\{0, 1\}^\nat\bigr) \oplus_\infty C\bigl(\{0, 1\}^\nat\bigr)
\to
C\bigl(\{0, 1\}^\nat\bigr)
\]
defined by formula~\eqref{eq:cantor-gamma-def} is an isometric
isomorphism, so $\bigl(C(\{0, 1\}^\nat), I, \gamma\bigr)$ is an object of
$\cat{A}^\infty$.

\begin{propn}[Universal property of functions on Cantor space]
\label{propn:Ainfty}
$\bigl(C(\{0, 1\}^\nat), I, \gamma)$ is the initial object of
$\cat{A}^\infty$. 
\end{propn}

\begin{proof}
For $n \geq 0$, let $\dyfn_n$ be the subspace of $C(\{0,
1\}^\nat)$ consisting of the functions $f$ such that for all $x =
(x_0, x_1, \ldots)$ and $y = (y_0, y_1, \ldots)$ in $\{0, 1\}^\nat$,
\[
x_i = y_i \text{ for all } i < n 
\implies
f(x) = f(y).
\]
Then $E_0$ is the linear span of $I$ and $E_{n + 1} = \gamma(E_n \oplus
E_n)$ for each $n \geq 0$. The argument used to prove Theorem~\ref{thm:Ap}
also shows that $\bigl(C(\{0, 1\}^\nat), I, \gamma)$ is initial in
$\cat{A}^\infty$, as long as $\dyfn = \bigcup_{n \geq 0} \dyfn_n$ is
dense in $C(\{0, 1\}^\nat)$. We show this now.

The topology on $\{0, 1\}^\nat$ is metrized by $d(x, y) = 2^{-\!\min\{n
\colon x_n \neq y_n\}}$. For $n \geq 0$, define $\pi_n \from \{0, 1\}^\nat
\to \{0, 1\}^\nat$ by
\[
\pi_n(x) = (x_0, \ldots, x_{n - 1}, 0, 0, \ldots).
\]
Thus, $d(x, \pi_n(x)) \leq 2^{-n}$ for all $x$. Now let $f \in C(\{0,
1\}^\nat)$; we prove that $f \in \overline{\dyfn}$. For each $n \geq 0$, we
have $f \of \pi_n \in \dyfn_n$ and
\[
\| f - f\of \pi_n \|_\infty 
\leq
\sup_{x, y:\, d(x, y) \leq 2^{-n}} |f(x) - f(y)|.
\]
But since $\{0, 1\}^\nat$ is compact, $f$ is uniformly continuous, so the
right-hand side converges to $0$ as $n \to \infty$. Hence $f = \lim_{n \to
\infty} f \of \pi_n \in \overline{\dyfn}$, as required.
\end{proof}

The next two results show how the universal property induces some standard
inclusions between function spaces. First let $1 \leq p \leq r < \infty$.
For any Banach spaces $V$ and $W$, the identity map
\[
V \oplus_r W \to V \oplus_p W
\]
is a contraction, by the elementary fact that power means are increasing in
their order \cite[Theorem~16]{HLP}. Hence $\cat{A}^p$ is a
subcategory of $\cat{A}^r$, and in particular, $(L^p[0, 1], I, \gamma)$ can
be regarded as an object of $\cat{A}^r$. By Theorem~\ref{thm:Ap}, there is
a unique map $(L^r[0, 1], I, \gamma) \to (L^p[0, 1], I, \gamma)$ in
$\cat{A}^r$.

\begin{propn}
\label{propn:A-incl}
For $1 \leq p \leq r < \infty$, the unique map $ (L^r[0, 1], I, \gamma) \to
(L^p[0, 1], I, \gamma) $ in $\cat{A}^r$ is the inclusion $L^r[0, 1] \incl
L^p[0, 1]$.
\end{propn}

\begin{proof}
Clearly the inclusion is one such map, and by Theorem~\ref{thm:Ap}, it is
the only one.
\end{proof}

Similarly, $\cat{A}^p$ is a subcategory of $\cat{A}^\infty$ for every $p$,
so $(L^p[0, 1], I, \gamma)$ is an object of $\cat{A}^\infty$. Hence by
Proposition~\ref{propn:Ainfty}, there is a unique map
\begin{equation}
\label{eq:CL}
\bigl( C( \{0, 1\}^\nat), I, \gamma \bigr)
\to 
(L^p[0, 1], I, \gamma)
\end{equation}
in $\cat{A}^\infty$. On the other hand, composition with the map $s \from
\{0, 1\}^\nat \to [0, 1]$ of
Remark~\ref{rmks:cantor}\bref{rmk:cantor-replace} defines a map $s^* \from
C[0, 1] \to C(\{0, 1\}^\nat)$.

\begin{propn}
\label{propn:A-incl-infty}
Let $1 \leq p < \infty$. The composite of $s^*$ with the unique
map~\eqref{eq:CL} in $\cat{A}^\infty$ is the inclusion $C[0, 1] \incl L^p[0,
1]$.
\end{propn}

\begin{proof}
Let $i \from [0, 1] \to \{0, 1\}^\nat$ be any section of the surjection $s$
(a choice of binary expansion of each element of $[0, 1]$). One easily
checks that $f \mapsto f\of i$ is a map of the form~\eqref{eq:CL} in
$\cat{A}^\infty$, so it is the unique such map. Hence the composite of
$s^*$ with~\eqref{eq:CL} is the map $C[0, 1] \to L^p[0, 1]$ given by $g
\mapsto g \of s \of i = g$.
\end{proof}

This result relates the abstractly characterized space $L^p[0, 1]$ to the
concrete function space $C[0, 1]$.

Next we use universal properties to construct the multiplication map
\begin{equation}
\label{eq:mult}
L^p[0, 1] \times L^q[0, 1] \toby{\cdot} L^1[0, 1]
\end{equation}
for each $p, q > 1$ such that $\tfrac{1}{p} + \tfrac{1}{q} = 1$.  
Composing~\eqref{eq:mult} with the integration operator $\int_0^1\from
L^1[0, 1] \to \F$ (derived in Proposition~\ref{propn:A-integration}), we
obtain the standard pairing between $L^p[0, 1]$ and $L^q[0, 1]$.  

For the rest of this section, write $L^p$ as shorthand for $L^p[0, 1]$.
For Banach spaces $V$ and $W$, write $\HOM(V, W)$ for the Banach space of
all bounded linear maps from $V$ to $W$, with the operator norm.

To construct the multiplication map~\eqref{eq:mult}, we will give
$\HOM(L^q, L^1)$ the structure of an object of $\cat{A}^p$.  By
Theorem~\ref{thm:Ap}, this structure will induce a map $L^p \to \HOM(L^q,
L^1)$, or equivalently a map $L^p \times L^q \to L^1$.  We will show that
this is multiplication.

To give $\HOM(L^q, L^1)$ the structure of an object of $\cat{A}^p$, first
recall that we have already constructed the inclusion $j \from L^q \incl
L^1$ (Proposition~\ref{propn:A-incl}).  This $j$ is a map in $\Ban$, that
is, an element of the closed unit ball of $\HOM(L^q, L^1)$.

Now define a linear map
\[
\Gamma \from 
\HOM(L^q, L^1) \oplus_p \HOM(L^q, L^1) \to \HOM(L^q, L^1)
\]
as follows: for $\phi_1, \phi_2 \in \HOM(L^q, L^1)$, let $\Gamma(\phi_1,
\phi_2)$ be the composite
\[
L^q \ltoby{\gamma^{-1}} 
L^q \oplus L^q \ltoby{\phi_1 \oplus \phi_2} 
L^1 \oplus L^1 \toby{\gamma}
L^1.
\]

\begin{lemma}
\label{lemma:hom-obj}
$(\HOM(L^q, L^1), j, \Gamma)$ is an object of $\cat{A}^p$.
\end{lemma}

\begin{proof}
It is immediate from the definitions that $\Gamma(j, j) = j$, so we only
have to show that $\Gamma$ is a contraction. Let $g \in L^q$. Writing
$\gamma^{-1}(g) = (g_1, g_2)$,
\begin{align}
\| (\Gamma(\phi_1, \phi_2))(g)\|        &
=
\| \gamma(\phi_1(g_1), \phi_2(g_2) \|
\nonumber       \\
&
=
\hlf\bigl(\|\phi_1(g_1)\| + \|\phi_2(g_2)\|\bigr)
\label{eq:Gamma-iso}    \\
&
\leq
\hlf(\|\phi_1\| \|g_1\| + \|\phi_2\| \|g_2\|)   
\nonumber       \\
&
\leq
\hlf
(\|\phi_1\|^p + \|\phi_2\|^p)^{1/p}
(\|g_1\|^q + \|g_2\|^q)^{1/q}
\label{eq:Gamma-hoelder}        \\
&
=
\Bigl( \hlf \bigl( \|\phi_1\|^p + \|\phi_2\|^p \bigr) \Bigr)^{1/p}
\Bigl( \hlf \bigl( \|g_1\|^q + \|g_2\|^q \bigr) \Bigr)^{1/q}
\nonumber       \\
&
=
\|(\phi_1, \phi_2)\| \|g\|
\label{eq:Gamma-final}
\end{align}
where~\eqref{eq:Gamma-iso} holds because $\gamma \from L^1 \oplus_1 L^1 \to
L^1$ is an isometry, \eqref{eq:Gamma-hoelder} is by H\"older's inequality
for vectors in $\R^2$, and~\eqref{eq:Gamma-final} holds because $\gamma
\from L^q \oplus_q L^q \to L^q$ is an isometry. Hence $\Gamma$ is a
contraction, as claimed.
\end{proof}

By Theorem~\ref{thm:Ap}, then, there is a unique map
\begin{equation}
\label{eq:pair-map-type}
(L^p, I, \gamma) \to (\HOM(L^q, L^1), j, \Gamma)
\end{equation}
in $\cat{A}^p$.  As a linear map $L^p \to \HOM(L^q, L^1)$, it corresponds
to a bilinear map 
\[
\square \from L^p \times L^q \to L^1.
\]

\begin{propn}
\label{propn:conj-prod}
The map $\square \from L^p \times L^q \to L^1$ is the product 
$(f, g) \mapsto f\cdot g$.
\end{propn} 

\begin{proof}
By H\"older's inequality for functions on $[0, 1]$, there is a linear
contraction 
\[
\begin{array}{cccc}
\theta\from     &L^p    &\to            &\HOM(L^q, L^1),        \\
                &f      &\mapsto        &f\cdot\dashbk.
\end{array}
\]
By the uniqueness part of Theorem~\ref{thm:Ap}, it suffices to show that
$\theta$ is a map of the form~\eqref{eq:pair-map-type} in $\cat{A}^p$.
Evidently $\theta(I) = j$.  It remains to show that the square
\[
\xymatrix@C+2em{
L^p \oplus L^p 
\ar[r]^-\gamma
\ar[d]_{\theta \oplus \theta}
&
L^p
\ar[d]^\theta
\\
\HOM(L^q, L^1) \oplus \HOM(L^q, L^1)    
\ar[r]_-\Gamma
&
\HOM(L^q, L^1)
}
\]
commutes, or equivalently that for all $f_1, f_2 \in L^p$,
\[
\Gamma(\theta(f_1), \theta(f_2)) 
=
\theta(\gamma(f_1, f_2)).
\]
This in turn is equivalent to 
\[
\gamma(f_1\cdot g_1, f_2\cdot g_2) 
= 
\gamma(f_1, f_2) \cdot \gamma(g_1, g_2)
\]
for all $f_1, f_2 \in L^p$ and $g_1, g_2 \in L^q$, which follows from the
definition of $\gamma$.
\end{proof}

The universal property of $L^p[0, 1]$ arises from the self-similarity of
$[0, 1]$, or more specifically, the isomorphism between $[0, 1]$ and two
copies of itself glued end to end. An analogous universal property is
satisfied by $L^p(X)$ for other self-similar spaces $X$. For example, the
isomorphism $\nat \iso \{*\} \amalg \nat$ gives rise to a universal
property characterizing the sequence space $\ell^p$, as follows.

Let $p \in [1, \infty]$. For Banach spaces $V$ and $W$, denote by $V
\dsl{p} W$ the direct sum $V \oplus W$ with norm
\[
\|(v, w)\| 
= 
\begin{cases}
\bigl( \|v\|^p + \|w\|^p \bigr)^{1/p}   &\text{if } p < \infty, \\
\max\{ \|v\|, \|w\| \}                  &\text{if } p = \infty.
\end{cases}
\]
Let $\Seqcat^p$ be the category of Banach spaces $V$ equipped with a map
$\delta \from \F \dsl{p} V \to V$ in $\Ban$. The maps $(V', \delta') \to
(V, \delta)$ in $\Seqcat^p$ are the maps $\theta \from V' \to V$ in $\Ban$
such that $\theta \of \delta' = \delta \of (\id_\F \oplus \theta)$.

Define $\gamma \from \F \dsl{p} \ell^p \to \ell^p$ by
\begin{equation}
\label{eq:seq-gamma}
\gamma\bigl(a, (a_0, a_1, \ldots)\bigr) 
= 
(a, a_0, a_1, \ldots).
\end{equation}
Then $\gamma$ is an isometric isomorphism, so $(\ell^p, \gamma)$ is an
object of $\Seqcat^p$. Write $c_0$ for the Banach space of sequences in
$\F$ converging to $0$, with the sup norm; then there is a map $\gamma
\from \F \dsl{\infty} c_0 \to c_0$ defined by~\eqref{eq:seq-gamma}, and
$(c_0, \gamma)$ is an object of $\Seqcat^\infty$.

\begin{propn}[Universal properties of sequence spaces]
\label{propn:seq-spaces}
$(\ell^p, \gamma)$ is the initial object of $\Seqcat^p$ for $1 \leq p < \infty$,
and $(c_0, \gamma)$ is the initial object of $\Seqcat^\infty$.
\end{propn}

\begin{proof}
We assume that $p < \infty$, the case $p = \infty$ being very similar.

For $n \geq 0$, let $\dyfn_n \sub \ell^p$ be the subspace of sequences of
the form $(a_0, \ldots, a_n, 0, 0, \ldots)$.  Then $\bigcup_{n \geq 0}
\dyfn_n$ is dense in $\ell^p$, so $\ell^p$ is the colimit of the diagram
$\dyfn_0 \incl \dyfn_1 \incl \cdots$ in $\Ban$. Moreover, $\gamma\from \F
\dsl{p} \ell^p \to \ell^p$ restricts to an isomorphism $\F \dsl{p} \dyfn_n
\to \dyfn_{n + 1}$ for each $n \geq 0$.

Let $(V, \delta) \in \Seqcat^p$.  To prove the existence of a map $\theta
\from (\ell^p, \gamma) \to (V, \delta)$ in $\Seqcat^p$, we define $\theta_n
\from \dyfn_n \to V$ inductively by $\theta_0(a_0, 0, 0, \ldots) =
\delta(a_0, 0)$ and
\[
\theta_{n + 1}
=
\Bigl( 
\dyfn_{n + 1} \ltoby{\gamma^{-1}}
\F \dsl{p} E_n \ltoby{\id_\F \oplus \theta_n} 
\F \dsl{p} V \toby{\delta}
V
\Bigr).
\]
Then $\theta_{n + 1}$ extends $\theta_n$ for each $n$, and by the colimit
description of $\ell^p$, there is a unique map $\theta \from \ell^p \to V$
extending every $\theta_n$.  The rest of the proof is similar to that of
Theorem~\ref{thm:Ap}, and further details are omitted.
\end{proof}

\begin{remark}
\label{rmk:alg-endo}
Theorem~\ref{thm:Ap} on $L^p[0, 1]$, Proposition~\ref{propn:Ainfty} on
$C(\{0, 1\}^\nat)$, and Proposition~\ref{propn:seq-spaces} on
$\ell^p$ and $c_0$ are all instances of a general categorical theorem that
not only proves the existence of these initial objects, but also constructs
them.

Let $\cat{D}$ be a category. An \demph{algebra} for an endofunctor $T \from
\cat{D} \to \cat{D}$ is an object $D$ of $\cat{D}$ together with a map
$\delta \from T(D) \to D$. A \demph{map} of algebras $(D',
\delta') \to (D, \delta)$ is a map $\theta \from D' \to D$ in $\cat{D}$
such that $\theta \of \delta' = \delta \of T\theta$. A lemma of Lambek
states that if $(D, \delta)$ is an initial algebra then $\delta$ is an
isomorphism \cite[Lemma~2.2]{LambFTC}.

Now suppose that $\cat{D}$ itself has an initial object, $Z$, that the
diagram
\[
Z \toby{!} 
T(Z) \ltoby{T(!)} 
T^2(Z) \ltoby{T^2(!)} 
\ \cdots
\]
has a colimit (where $!$ is the unique map $Z \to T(Z)$), and that
this colimit is preserved by $T$; in other words, the canonical map
\[
\eta\from 
\colim_n T(T^n(Z)) 
\to
T \bigl( \colim_n T^n(Z) \bigr)
\]
is an isomorphism.  A theorem of Ad\'amek states that $\bigl(
\colim\limits_n T^n(Z), \eta^{-1} \bigr)$ is the initial $T$-algebra
(\cite[p.~591]{AdamFAA} or \cite[Theorem~3.17]{AdamIC}).

For example, let $\cat{D}$ be the category $\Ban_*$ of pairs $(V, v)$ where
$V \in \Ban$ and $v \in V$ with $\|v\| \leq 1$; maps in
$\Ban_*$ are maps in $\Ban$ preserving the chosen points.  (This is the
coslice category $\F/\Ban$.)  Define $T \from \cat{D} \to \cat{D}$ by
\[
T(V, v) = \bigl(V \oplus_p V,\, (v, v)\bigr).
\]
Then the category of $T$-algebras is $\cat{A}^p$. Moreover, $(\F, 1$) is
initial in $\Ban_*$ and the hypotheses of Ad\'amek's theorem hold, so the
initial object of $\cat{A}^p$ is constructed as the colimit over $n \geq 0$
of the objects $T^n(\F, 1)$ of $\Ban_*$.  Concretely, $T^n(\F, 1) =
(\dyfn_n, I)$, where $\dyfn_n$ is as defined in the proof of
Theorem~\ref{thm:Ap}. Hence this categorical method constructs $L^p[0, 1]$
as the colimit of $\dyfn_0 \incl \dyfn_1 \incl \cdots$ in
$\Ban$. Similarly, Propositions~\ref{propn:Ainfty}
and~\ref{propn:seq-spaces} are also instances of Ad\'amek's theorem. In
each of these results, the map denoted by $\gamma$ is an isometric
isomorphism, and Lambek's lemma explains why.
\end{remark}

\section{Integration on an arbitrary measure space}
\label{sec:B}

Our second main theorem characterizes the functor $L^p$ from measure spaces
to Banach spaces, again by a simple universal property.  

The measure on a measure space $X$ is written as $\mu_X$. Throughout,
\emph{all measure spaces are understood to be finite}, in the sense that
$\mu_X(X) < \infty$.

An \demph{embedding} $i \from Y \to X$ of measure spaces is an injection
such that $B \sub Y$ is measurable if and only if $i B \sub X$ is
measurable, and in that case, $\mu_Y(B) = \mu_X(i B)$. In particular, we
deviate from common usage of the term `embedding' by requiring $iY$ to be a
measurable subset of $X$. Measure spaces and embeddings form a category
$\Memb$.

A \demph{measure-preserving partial map} $(A, s) \from X \to Y$ is a
measurable subset $A \sub X$ together with a measure-preserving map $s
\from A \to Y$. Here $A$ is given the unique measure space structure such
that the inclusion $A \incl X$ is an embedding. The composite of 
measure-preserving partial maps
\[
X \ltoby{(A, s)} 
Y \ltoby{(B, t)} 
Z
\]
is $(s^{-1}B, u)$, where $u\from s^{-1}B \to Z$ is defined by $u(x) =
t(s(x))$.  Measure spaces and measure-preserving partial maps form a
category $\Mpar$.

For each $p \in [1, \infty]$, there is a functor 
\[
L^p \from \Mpar^\op \to \Ban,
\]
defined on objects in the usual way and on maps as follows: given a
measure-preserving partial map $(A, s) \from X \to Y$, the induced 
map $L^p(Y) \to L^p(X)$ is $g \mapsto (g \of s)^X$, where $(g \of s)^X$
denotes the composite $g \of s \from A \to \F$ extended by zero to $X$.

Any embedding $i \from Y \to X$ determines a map $(iY, i^{-1}) \from X \to
Y$ in $\Mpar$; note the change of direction. (In particular, every
measurable subset $A$ of $X$ gives a map $(A, \id_A) \from X \to A$.) Also,
every measure-preserving map $s \from X \to Y$ determines a map $(X, s)
\from X \to Y$ in $\Mpar$. Thus, both $\Memb^\op$ and the category $\Mpres$
of measure spaces and measure-preserving maps can be viewed as
subcategories of $\Mpar$:
\[
\xymatrix@R-6.5ex{
\Memb^\op  \ar@{^{(}->}[rd]     &               \\
                                &\Mpar.         \\
\Mpres     \ar@{^{(}->}[ru]     &
}
\]
Furthermore, every measure-preserving partial map factors canonically into
maps of these two special types:
\[
X \ltoby{(A, s)} Y
=
\Bigl(
X \ltoby{(A, \id_A)} 
A \ltoby{(A, s)}
Y
\Bigr).
\]
This has two consequences.  First, any functor $F \from \Mpar^\op \to
\cat{C}$ to any category $\cat{C}$ is determined by its effect on
objects, embeddings, and measure-preserving maps. For simplicity, we
write $Fi$ for the map $F(iY, i^{-1}) \from F(Y) \to F(X)$ induced by an
embedding $i \from Y \to X$, and $Fs$ for the map $F(X, s) \from F(Y) \to
F(X)$ induced by a measure-preserving map $s \from X \to Y$.  Second, a
transformation between two functors $\Mpar^\op \to \cat{C}$ is natural
if and only if it is natural with respect to both embeddings and
measure-preserving maps.

\begin{lemma}
\label{lemma:bc}
Let $F \from \Mpar^\op \to \cat{C}$ be a functor into a category
$\cat{C}$. Let $s \from X \to Y$ be a measure-preserving map and $B$ a
measurable subset of $Y$. Write
\[
\xymatrix@M+.5ex{
s^{-1} B 
\ar[r]^-{s'} 
\ar@{^{(}->}[d]_i       &
B
\ar@{^{(}->}[d]^j       \\
X 
\ar[r]_s        &
Y
}
\]
for the resulting inclusions ($i$ and $j$) and the restriction $s'$ of
$s$. Then the square
\[
\xymatrix{
F(s^{-1}B) 
\ar[d]_-{Fi}     &
F(B)
\ar[l]_-{Fs'} 
\ar[d]^{Fj}     \\
F(X)    &
F(Y) \ar[l]^{Fs}
}
\]
in $\cat{C}$ commutes.
\end{lemma}
Results of this type are known as Beck--Chevalley
conditions~\cite[Section~IV.9]{MaMo}.

\begin{proof}
The square
\[
\xymatrix@C+4ex{
s^{-1} B
\ar[r]^{(s^{-1}B, s')}  &
B       \\
X 
\ar[u]^{(s^{-1}B, \id_{s^{-1}B})}
\ar[r]_{(X, s)} &
Y
\ar[u]_{(B, \id_B)}
}
\]
in $\Mpar$ commutes, and the result follows by functoriality of $F$.
\end{proof}

\begin{table}
\hspace*{-5.5mm}%
\begin{tabular}{lllll}
\hline
\vphantom{$l^{l^l}$}%
Result  &Categories     &Type of space  &Initial        &Type of        \\
        &               &               &object         &function\\
\hline  
\vphantom{$l^{l^l}$}%
Proposition~\ref{propn:BV}      &
$\BVpar$, $\BVemb$     &
Vector space            &
$(\Simp, I)$    &
simple
\\[1mm]
Proposition~\ref{propn:BN}     &
$\BNpar^p$, $\BNemb^p$  &
Normed vector space     &
$(\simp^p, I)$  &
simple$/\!\!=_{\text{a.e.}}$
\\[1mm]
Theorem~\ref{thm:Bp}    &
$\Bpar^p$, $\Bemb^p$    &
Banach space            &
$(L^p, I)$      &
$p$-integrable$/\!\!=_{\text{a.e.}}$    \\
\hline
\end{tabular}
\caption{Results leading up to the main theorem, Theorem~\ref{thm:Bp}. For
example, Proposition~\ref{propn:BV} concerns two categories, $\BVpar$ and
$\BVemb$, whose objects are vector-space-valued functors together with some
extra data; it states that the initial object of both categories is
$(\Simp, I)$, where $\Simp$ denotes the simple functions.}
\label{table:B-results}
\end{table}

We prove our main theorem in three steps (Table~\ref{table:B-results}),
which informally are as follows. First, the universal \emph{vector space}
obtained from a measure space consists of the simple functions. Second, the
universal \emph{normed vector space} on a measure space consists of the
almost everwhere equivalence classes of simple functions. Third, as the
main theorem, the universal \emph{Banach space} on a measure space consists
of the almost everywhere equivalence classes of integrable functions.  Each
step has two versions, according to whether one considers functoriality
with respect to all measure-preserving partial maps or only the embeddings.

We now begin the first step. 
Two embeddings
\[
Y \toby{i} X \otby{j} Z
\]
of measure spaces are \demph{complementary} if $iY \cap jZ = \emptyset$ and
$iY \cup jZ = X$.
Write $\Vect$ for the category of vector spaces. Let $\BVemb$ be the
category of pairs $(F, v)$ consisting of a functor $F\from \Memb \to \Vect$
together with an element $v_X \in F(X)$ for each measure space $X$,
satisfying the following axiom:
\begin{itemize}
\item[\KKcomp] 
$(Fi)(v_Y) + (Fj)(v_Z) = v_X$ for all pairs of complementary
embeddings $Y \toby{i} X \otby{j} Z$.
\end{itemize}
A map $(F', v') \to (F, v)$ in $\BVemb$ is a natural transformation $\psi
\from F' \to F$ such that $\psi_X(v'_X) = v_X$ for all $X$. 

When $i \from Y \to X$ is an embedding and $i$ is understood, we write
$(Fi)(v_Y) \in F(X)$ as $v^X_Y$. In this notation, axiom~\Kcomp\ states
that $v^X_Y + v^X_Z = v_X$.

For example, there is an object $(\Simp, I)$ of $\BVemb$ defined as
follows. The functor $\Simp \from \Memb \to \Vect$ assigns to each measure
space $X$ the space of simple functions $X \to \F$, and is defined on
embeddings by extending simple functions by zero. The element $I_X \in
\Simp_X$ is the function with constant value $1$. Given an embedding $i
\from Y \to X$, the simple function $I^X_Y = (\Simp i)(I_Y)$ on $X$ is the
indicator function of $iY \sub X$. Axiom~\Kcomp\ is the elementary fact
that $I^X_Y + I^X_Z = I_X$ for any measurable partition $X = Y \amalg Z$.

To incorporate functoriality with respect to all measure-preserving partial
maps, let $\BVpar$ be the category of pairs $(F, v)$ consisting of a
functor $F \from \Mpar^\op \to \Vect$ together with an element $v_X \in
F(X)$ for each measure space $X$, satisfying axiom~\Kcomp\ and:
\begin{itemize}
\item[\KKpres] 
$(Fs)(v_Y) = v_X$ for all measure-preserving maps $X \toby{s} Y$.
\end{itemize}
The maps in $\BVpar$ are defined as in $\BVemb$. 

The functor $\Simp$ on $\Memb$ extends naturally to a functor $\Simp \from
\Mpar^\op \to \Vect$, since the composite of a simple function $Y \to \F$
with a measure-preserving map $X \to Y$ is again simple.  The pair $(\Simp
\from \Mpar^\op \to \Vect, I)$ is an object of $\BVpar$. By abuse of
notation, we write $(\Simp, I)$ for both this object of $\BVpar$ and the
object of $\BVemb$ defined previously, which is its image under the evident
forgetful functor $\BVpar \to \BVemb$.

The following two lemmas establish basic properties of $\BVemb$ and
$\BVpar$. 

\begin{lemma}
\label{lemma:F-additivity}
Let $(F, v) \in \BVemb$.  Let $n \geq 0$ and let $(i_r \from Y_r \to X)_{1
\leq r \leq n}$ be a family of embeddings with pairwise disjoint
images. Then
\[
v^X_{i_1 Y_1 \cup \cdots \cup i_n Y_n}
=
v^X_{Y_1} + \cdots + v^X_{Y_n}.
\]
In particular, $v^X_\emptyset = 0$ for any measure space $X$.
\end{lemma}

\begin{proof}
Put $Y = i_1 Y_1 \cup \cdots \cup i_n Y_n \sub X$.  First we show that
\begin{equation}
\label{eq:F-add-basic}
v_Y = v^Y_{Y_1} + \cdots + v^Y_{Y_n}.
\end{equation}
Axiom~\Kcomp\ applied to the embeddings $\emptyset \to \emptyset \ot
\emptyset$ gives $v_\emptyset + v_\emptyset = v_\emptyset$ and so
$v_\emptyset = 0$, which is the case $n = 0$.  The general case follows by
induction, again using~\Kcomp.  This proves~\eqref{eq:F-add-basic}.  Now
the functoriality of $F \from \Memb \to \Vect$ gives
\[
v^X_Y 
=
\bigl( v^Y_{Y_1} + \cdots + v^Y_{Y_n} \bigr)^{X}
=
v^X_{Y_1} + \cdots + v^X_{Y_n},
\]
as required. The last part of the lemma is the case $n = 0$.
\end{proof}

\begin{lemma}
\label{lemma:v-pres}
Let $(F, v) \in \BVpar$. Let $s\from X \to Y$ be a measure-preserving map
and $B$ a measurable subset of $Y$. Then $(Fs)(v^Y_B) = v^X_{s^{-1}B}$.
\end{lemma}

\begin{proof}
Define $i$, $j$ and $s'$ as in Lemma~\ref{lemma:bc}. We have 
\[
(Fs)(v^Y_B) = (Fs)(Fj)(v_B)
\]
by definition of $v^Y_B$, and 
\[
v^X_{s^{-1}B} = (Fi)(v_{s^{-1}B}) = (Fi)(Fs')(v_B)
\]
by definition of $v^X_{s^{-1}B}$ and axiom~\Kpres. The result follows from
Lemma~\ref{lemma:bc}.
\end{proof}

We now establish the universal properties of the simple functions. The
proof implicitly uses the fact that a function on a measure space is simple
if and only if its image is finite and each fibre is measurable.

\begin{propn}[Universal properties of spaces of simple functions]
\label{propn:BV}
\
\begin{enumerate}
\item
\label{part:BV-emb}
$(\Simp, I)$ is the initial object of $\BVemb$.

\item
\label{part:BV-par}
$(\Simp, I)$ is the initial object of $\BVpar$.
\end{enumerate}
\end{propn}

\begin{proof}
For~\bref{part:BV-emb}, let $(F, v) \in \BVemb$.  We show that there
exists a unique map $\BVmap\from (\Simp, I) \to (F, v)$ in $\BVemb$.

\paragraph*{Uniqueness}
Let $\BVmap \from (\Simp, I) \to (F, v)$ in $\BVemb$.  For each
measure space $X$ and measurable $Y \sub X$, the naturality of $\BVmap$
with respect to the inclusion $i\from Y \incl X$ gives a commutative square
\[
\xymatrix{
\Simp(Y) 
\ar@{^{(}->}[r]
\ar[d]_{\BVmap_Y}      &
\Simp(X)
\ar[d]^{\BVmap_X}      \\
F(Y)
\ar[r]_{Fi}     &
F(X).
}
\]
Evaluating the square at $I_Y \in \Simp(Y)$ gives 
\[
\BVmap_X(I^X_Y)
=
(Fi)(\BVmap_Y(I_Y))
=
(Fi)(v_Y)
=
v^X_Y.
\]
Hence $\BVmap_X$ is uniquely determined on indicator functions of
measurable subsets of $X$, and therefore, by linearity, on all of
$\Simp(X)$.

\paragraph*{Existence}
For each measure space $X$ and $f \in \Simp(X)$, define
\begin{equation}
\label{eq:defn-BVmap}
\BVmap_X(f)
=
\sum_{c \in \F} cv^X_{f^{-1}(c)} 
\in 
F(X).
\end{equation}
This sum has only finitely many nonzero summands, since if $c$ is not in
the image of $f$ then $v^X_{f^{-1}(c)} = v^X_\emptyset = 0$ by
Lemma~\ref{lemma:F-additivity}. 

To show that $\BVmap_X\from \Simp(X) \to F(X)$ is linear, let $f, g \in
\Simp(X)$. Then
\begin{equation}
\label{eq:lin-add}
\BVmap_X(f + g) 
=
\sum_c c v^X_{(f + g)^{-1}(c)}.
\end{equation}
But $(f + g)^{-1}(c) = \coprod_{a, b \colon a + b = c} f^{-1}(a) \cap
g^{-1}(b)$, so by Lemma~\ref{lemma:F-additivity},
\[
v^X_{(f + g)^{-1}(c)}
=
\sum_{a, b \colon a + b = c} v^X_{f^{-1}(a) \cap g^{-1}(b)}.
\]
Hence by~\eqref{eq:lin-add},
\begin{align}
\BVmap_X(f + g) &
=
\sum_a a \sum_b v^X_{f^{-1}(a) \cap g^{-1}(b)}     
+
\sum_b b \sum_a v^X_{f^{-1}(a) \cap g^{-1}(b)}     
\nonumber       \\
&
=
\sum_a a v^X_{f^{-1}(a)} + \sum_b b v^X_{g^{-1}(b)}     
\label{eq:lin-fibres}   \\
&
=
\BVmap_X(f) + \BVmap_X(g),
\nonumber
\end{align}
where~\eqref{eq:lin-fibres} is obtained by applying
Lemma~\ref{lemma:F-additivity} to the disjoint unions
\[
f^{-1}(a) = \coprod_b f^{-1}(a) \cap g^{-1}(b),
\qquad
g^{-1}(b) = \coprod_a f^{-1}(a) \cap g^{-1}(b).
\]
We have now shown that $\BVmap_X(f + g) = \BVmap_X(f) + \BVmap_X(g)$ for all
$f, g \in \Simp(X)$.  A similar argument shows that $\BVmap_X$ preserves
scalar multiplication.  Thus, $\BVmap_X$ is a linear map $\Simp(X) \to
F(X)$.

Next we show that $\BVmap_X$ is natural in $X \in \Memb$.  That is, given
an embedding $i\from Y \to X$, we show that the square
\[
\xymatrix@M+.5ex{
\Simp(Y) 
\ar@{^{(}->}[r]
\ar[d]_{\BVmap_Y} &
\Simp(X)
\ar[d]^{\BVmap_X}       \\
F(Y) 
\ar[r]_{Fi}     &
F(X)
}
\]
commutes.  By linearity, it suffices to check this on the indicator
function $I^Y_B$ of a measurable subset $B \sub Y$.  On the one hand,
$\BVmap_Y(I^Y_B) = v^Y_B$ by definition of $\BVmap_Y$, so
$(Fi)((\BVmap_Y(I^Y_B)) = (v^Y_B)^X = v^X_B$.  On the other, the extension
by zero of $I^Y_B$ to $X$ is $I^X_B$, and $\BVmap_X(I^X_B) = v^X_B$.  The
square therefore commutes.

We have now shown that $\BVmap = (\BVmap_X)$ defines a natural
transformation $\Simp \to F$.  Moreover, $\BVmap_X(I_X) = v_X$ for each
measure space $X$, by definition of $\BVmap_X$.  Hence $\BVmap$ is a map
$(\Simp, I) \to (F, v)$ in $\BVemb$, completing the proof
of~\bref{part:BV-emb}. 

To prove~\bref{part:BV-par}, let $(F, v) \in
\BVpar$. By~\bref{part:BV-emb}, there is a unique map $\BVmap \from (\Simp,
I) \to (F, v)$ in $\BVemb$, and it suffices to prove that $\BVmap$ is a map
in $\BVpar$. This reduces to showing that for every measure-preserving map
$s \from X \to Y$, the naturality square
\[
\xymatrix{
\Simp(Y)
\ar[r]^{\dashbk\of s}        
\ar[d]_{\BVmap_Y}       &
\Simp(X)
\ar[d]^{\BVmap_X}       \\
F(Y)    
\ar[r]_{Fs}     &
F(X)
}
\]
commutes. By linearity, it suffices to check this on the indicator
function $I^Y_B$ of a measurable subset $B \sub Y$.  On the one hand,
$\psi_X(I^Y_B \of s) = \psi_X(I^X_{s^{-1}B}) = v^X_{s^{-1}B}$. On the
other, $(Fs)(\psi_Y(I^Y_B)) = (Fs)(v^Y_B)$. Lemma~\ref{lemma:v-pres}
concludes the proof.
\end{proof}

This completes the first step towards the main theorem
(Table~\ref{table:B-results}). We now embark on the second.

Write $\NVS$ for the category of normed vector spaces and contractions. Let
$p \in [1, \infty]$. Let $\BNemb^p$ be the category of pairs $(F, v)$
consisting of a functor $F \from \Memb \to \NVS$ together with an element
$v_X \in F(X)$ for each measure space $X$, satisfying axiom~\Kcomp\ and the
following axioms:
\begin{itemize}
\item[\KKnorm] 
$\|v_X\| \leq \mu_X(X)^{1/p}$ for all measure spaces $X$;

\item[\KKineq]
$\|(Fi)(u) + (Fj)(w)\| \leq \bigl( \|u\|^p + \|w\|^p \bigr)^{1/p}$ for all
pairs $Y \toby{i} X \otby{j} Z$ of complementary embeddings,
$u \in F(Y)$, and $w \in F(Z)$.
\end{itemize}
The maps in $\BNemb^p$ are defined as in $\BVemb$.

When $p = 1$, axiom~\Kineq\ is redundant, since the maps in $\Ban$ are
contractions.  When $p = \infty$, the expression $\mu_X(X)^{1/p}$
in~\Knorm\ is interpreted as $0$ if $\mu_X(X) = 0$ and $1$ otherwise,
and the right-hand side of the inequality in~\Kineq\ as $\max\{ \|u\|,
\|w\| \}$.

Let $\BNpar^p$ be the category defined in the same way, but replacing
$\Memb$ by $\Mpar^\op$ and requiring that each object $(F, v)$ satisfies
axioms \Kcomp--\Kineq.

For example, an object $(S^p, I)$ of $\BNemb^p$ and an object $(S^p, I)$ of
$\BNpar^p$ may be defined as follows (abusively using the same name for
both).  For a measure space $X$, write $\simp^p(X)$ for the vector space of
equivalence classes of simple functions on $X$ under equality almost
everywhere, with the $L^p$ norm:
\begin{equation}
\label{eq:simple-norm}
\|f\|_p
=
\biggl( 
\sum_{c \in \F} |c|^p \mu_X\bigl(f^{-1}(c)\bigr)
\biggr)^{1/p}
\end{equation}
if $p < \infty$, and
\[
\|f\|_\infty
=
\max \bigl\{ 
|c| \such c \in \F, \ \mu_X\bigl(f^{-1}(c)\bigr) > 0 
\bigr\}.
\]
(All but finitely many arguments in the sum and max are $0$.)  This
construction defines functors
\[
\simp^p \from \Memb \to \NVS,
\qquad
\simp^p \from \Mpar^\op \to \NVS, 
\]
with the same functorial action as $\Simp$. They determine an object
$(\simp^p, I)$ of $\BNemb^p$ and an object $(\simp^p, I)$ of
$\BNpar^p$. Axioms \Kcomp--\Kineq\ are elementary properties of indicator
functions and the $p$-norm. Equality holds in the
inequalities~\Knorm\ and~\Kineq. Axiom~\Kineq\ is a consequence of the
fact that the $p$-norm of a function on a disjoint union $Y
\amalg Z$ is determined in the evident way by the $p$-norms of its
restrictions to $Y$ and $Z$.

The assumption that our measure spaces are finite has been used here
implicitly, since if $\mu_X(X)$ were infinite then the constant function
$I_X$ on $X$ would not lie in $L^p(X)$ for any $p \in [1, \infty)$.

\begin{lemma}
\label{lemma:BN-ineq}
Let $(F, v) \in \BNemb^p$. Then $\|v^X_Y\| \leq \mu_X(iY)^{1/p}$ for any
embedding $i \from Y \to X$ of measure spaces. In particular, $v^X_Y = 0$
whenever $Y$ is a measure-zero subspace of a measure space $X$.
\end{lemma}

\begin{proof}
We have $\|v^X_Y\| = \|(Fi)(v_Y)\| \leq \|v_Y\|$ since $Fi$ is a contraction,
and $\|v_Y\| \leq \mu_Y(Y)^{1/p} = \mu_X(iY)^{1/p}$ by axiom~\Knorm.
\end{proof}

\begin{propn}[Universal properties of spaces of simple functions]
\label{propn:BN}
Let $1 \leq p \leq \infty$.
\begin{enumerate}
\item 
\label{part:BN-emb}
$(\simp^p, I)$ is the initial object of $\BNemb^p$.

\item
\label{part:BN-par}
$(\simp^p, I)$ is the initial object of $\BNpar^p$.
\end{enumerate}
\end{propn}

\begin{proof}
For~\bref{part:BN-emb}, let $(F, v) \in \BNemb^p$. We show that there is a
unique map $(\simp^p, I) \to (F, v)$ in $\BNemb^p$. Write $\BVmap$ for the
unique map $(\Simp, I) \to (F, v)$ in $\BVemb$.

\paragraph{Uniqueness}
Let $\BNmap$ be a map $(\simp^p, I) \to (F, v)$ in $\BNemb^p$. Applying the
forgetful functor $\BNemb^p \to \BVemb$ gives a map $(\simp^p, I) \to (F,
v)$ in $\BVemb$, which we also denote by $\BNmap$. Its composite with the
quotient map $(\Simp, I) \to (\simp^p, I)$ in $\BVemb$ can only be
$\BVmap$, so $\BNmap$ is uniquely determined.

\paragraph{Existence}
First we prove that $\|\BVmap_X(f)\| \leq \|f\|_p$ for each measure space
$X$ and $f \in \Simp(X)$. If $p < \infty$ then
\begin{align}
\|\BVmap_X(f)\| &
=
\biggl\| \sum_{c \in \F} c v^X_{f^{-1}(c)} \biggr\|
\label{eq:BN-map-def}   \\
&
\leq
\biggl( \sum_{c \in \F} 
\bigl\|c v^X_{f^{-1}(c)} \bigr\|^p 
\biggr)^{1/p}
\label{eq:BN-iter-ineq}       \\
&
\leq
\biggl(
\sum_{c \in \F} |c| \mu_X(f^{-1}(c))
\biggr)^{1/p}
\label{eq:BN-norm-ineq} \\
&
=
\|f\|_p,
\label{eq:BN-norm-def}
\end{align}
where~\eqref{eq:BN-map-def} follows from the formula~\eqref{eq:defn-BVmap}
for $\psi$, inequality \eqref{eq:BN-iter-ineq} follows by induction from
axiom~\Kineq, \eqref{eq:BN-norm-ineq} follows from
Lemma~\ref{lemma:BN-ineq}, and~\eqref{eq:BN-norm-def} follows
from~\eqref{eq:simple-norm}. The proof for $p = \infty$ is similar. Hence
$\|\BVmap_X(f)\| \leq \|f\|_p$, as claimed.

It follows that for each measure space $X$, there is a unique linear map
$\BNmap_X$ such that the triangle
\[
\xymatrix{
\Simp(X) 
\ar@{->>}[r]
\ar[dr]_{\BVmap_X}      &
S^p(X) 
\ar@{..>}[d]^{\BNmap_X} \\
&
F(X)
}
\]
commutes (where the horizontal arrow is the quotient map), and that
$\BNmap_X$ is a contraction. Moreover, $\BNmap_X$ is natural in $X$ because
$\BVmap_X$ is, and $\BNmap_X(I_X) = \BVmap_X(I_X) = v_X$. Hence $\BNmap$ is
a map $(S^p, I) \to (F, v)$ in $\BNemb^p$, completing the proof
of~\bref{part:BN-emb}.

Part~\bref{part:BN-par} follows from the fact that $\psi$ is natural with
respect to measure-preserving maps, by
Proposition~\ref{propn:BV}\bref{part:BV-par}.
\end{proof}

We now come to the main theorem. Let $\Bemb^p$ denote the full subcategory
of $\BNemb^p$ consisting of the pairs $(F, v)$ such that $F$ takes values
in Banach spaces, and similarly $\Bpar^p \sub \BNpar^p$. Equivalently,
$\Bemb^p$ and $\Bpar^p$ are defined in the same way as $\BNemb^p$ and
$\BNpar^p$ respectively, but replacing $\NVS$ by $\Ban$.

We have already defined the functor $L^p \from \Mpar^\op \to \Ban$, and we
also write $L^p$ for the restricted functor $\Memb \to \Ban$. These
functors define an object $(L^p, I)$ of $\Bemb^p$ and an object $(L^p, I)$
of $\Bpar^p$.

\begin{thm}[Universal properties of the $L^p$ functors]
\label{thm:Bp}
Let $1 \leq p \leq \infty$.
\begin{enumerate}
\item 
\label{part:B-emb}
$(L^p, I)$ is the initial object of $\Bemb^p$.

\item
\label{part:B-par}
$(L^p, I)$ is the initial object of $\Bpar^p$.
\end{enumerate}
\end{thm}

Theorem~B of the Introduction is the case $p = 1$ of~\bref{part:B-par}.

\begin{proof}
The inclusion $\Ban \incl \NVS$ has a left adjoint, the completion
functor $V \mapsto \ovln{V}$. Given $(F, v) \in \Bemb^p$, define $\ovln{F}
\from \Mpar \to \Ban$ by $\ovln{F}(X) = \ovln{F(X)}$, and regard $v_X \in
F(X)$ as an element of $\ovln{F}(X)$. It is routine to verify that
$(\ovln{F}, v) \in \Bemb^p$ and that the forgetful functor $\Bemb^p \to
\BNemb^p$ has left adjoint $(F, v) \mapsto (\ovln{F}, v)$.

Left adjoints preserve initial objects, so by Proposition~\ref{propn:BN},
the initial object of $\Bemb^p$ is $(\ovln{S^p}, I)$. But $L^p(X)$ is the
completion of $S^p(X)$ for each $X$~\cite[Proposition~6.7 and
Theorem~6.8]{FollRA}, so the initial object of $\Bemb^p$ is $(L^p, I)$.

This proves~\bref{part:B-emb}, and the same argument
proves~\bref{part:B-par}. 
\end{proof}

Theorem~\ref{thm:Bp} characterizes the spaces $L^p(X)$ uniquely up to
\emph{isometric} isomorphism, since the maps in $\Ban$ are contractions.

We now focus on the case $p = 1$.  As in the case of the interval, the
universal characterization of the integrable functions yields a unique
characterization of integration, as follows.

Write $\F \from \Memb \to \Ban$ for the functor that sends all measure
spaces to the ground field $\F \in \Ban$ and all maps in $\Memb$ to
$\id_{\F}$.  For each measure space $X$, put $t_X = \mu_X(X) \in \F$.  Then
$(\F, t) \in \Bemb^1$.

\begin{propn}[Uniqueness of integration]
\label{propn:B-unique-int}
The unique map $(L^1, I) \to (\F, t)$ in $\Bemb^1$ is the family of
operators
\[
\int_\dashbk
=
\biggl( \int_X \from L^1(X) \to \F \biggr)_{\textup{measure spaces } X}.
\]
\end{propn}

\begin{proof}
By Theorem~\ref{thm:Bp}\bref{part:B-emb}, it suffices to show that
$\int_\dashbk$ is indeed a map $(L^1, I) \to (\F, t)$ in $\Bemb^1$. This
reduces to the following standard properties of integration.  First,
whenever $X$ is a measure space, $\int_X$ is a map in $\Bemb^1$ by linearity
of integration and the triangle inequality
\[
\left| \int_X f \right|       
\leq 
\int_X |f|
\]
($f \in L^1(X)$).  Second, $\int_\dashbk$ defines a natural transformation
$L^1 \to \F$ because
\[
\int_X g^X = \int_Y g
\]
for any embedding of measure spaces $Y \to X$ and $g \in L^1(Y)$, where
$g^X$ denotes the extension by zero of $g$ to $X$.  Finally,
\[
\int_X I_X = \mu_X(X)
\]
for every measure space $X$.
\end{proof}

\begin{remark}
Here we have treated $(\F, t)$ as an object of $\Bemb^1$. But the constant
functor $\F \from \Mpar^\op \to \Ban$ together with the elements $t_X \in
\F$ also define an object $(\F, t)$ of $\Bpar^1$.
Theorem~\ref{thm:Bp}\bref{part:B-par} implies that the unique map $(L^p, I)
\to (\F, t)$ in $\Bemb^1$ is in fact a map in $\Bpar^1$. In concrete terms,
this statement is the formula for integration under a change of variables:
\[
\int_Y g = \int_X g \of s
\]
whenever $s \from X \to Y$ is a measure-preserving map and $g \in L^1(Y)$.
\end{remark}

Our next result relates the abstractly characterized spaces $L^1(X)$ to
actual spaces of functions. Of course, we cannot hope to evaluate an
element of $L^1(X)$ at a point of $X$, since it is only an equivalence
class of integrable functions. The best we can hope is to be able to
integrate it over any measurable subset of $X$. That is, we would like to
construct for each $f \in L^1(X)$ the signed or complex measure $f\mu_X =
f\dee\mu_X$ defined by $(f \mu_X)(A) = \int_A f \dee\mu_X$. We now show
that this construction arises naturally from the universal property of
$L^1$.

Write $M(X)$ for the Banach space of finite signed measures (if $\F = \R$)
or complex measures (if $\F = \C$) on the underlying $\sigma$-algebra of a
measure space $X$, with the total variation norm $\nu \mapsto
|\nu|(X)$. Any embedding $i \from Y \to X$ induces an isometry $M(Y) \to
M(X)$ that extends measures by zero; thus, $M$ defines a functor $\Memb \to
\Ban$.  Together with the elements $\mu_X \in M(X)$, it gives an object
$(M, \mu)$ of $\Bemb^1$.

\begin{propn}
\label{propn:B-action}
The unique map $(L^1, I) \to (M, \mu)$ in $\Bemb^1$ has $X$-component
\[
\begin{array}{ccc}
L^1(X)  &\to            &M(X)   \\
f       &\mapsto        &f\mu_X
\end{array}
\]
for each measure space $X$.
\end{propn}

\begin{proof}
For each $X$, define $\theta_X\from L^1(X) \to M(X)$ by $\theta_X(f) =
f\mu_X$.  By Theorem~\ref{thm:Bp}\bref{part:B-emb}, it suffices to show that
$\theta$ defines a map $(L^1, I) \to (M, \mu)$ in $\Bemb^1$.  It is
elementary that $\theta_X$ is an isometry and that $\theta_X(I_X) =
\mu_X$, for each $X$.  So it only remains to prove that $\theta_X$ is
natural in $X \in \Memb$.

Let $i\from Y \to X$ be an embedding.  We must show that the square
\[
\xymatrix{
L^1(Y) 
\ar@{^{(}->}[r] 
\ar[d]_{\theta_Y}       &
L^1(X)
\ar[d]^{\theta_X}       \\
M(Y)
\ar@{^{(}->}[r]         &
M(X)
}
\]
commutes, where both horizontal maps are extension by zero. Equivalently,
writing $\blank^X$ for the extension by zero to $X$ of a function or
measure on $Y$, we must show that $g^X \mu_X = (g\mu_Y)^X$ for all $g \in
L^1(Y)$. But this just states that
\[
\int_A g^X \dee\mu_X = \int_{A \cap Y} g \dee\mu_Y
\]
for all measurable $A \sub X$, which is true.
\end{proof}

There is a similar theorem in which $M(X)$ is replaced by the subspace
$\AC(X)$ of signed or complex measures absolutely continuous with respect
to $\mu_X$; the unique map $(L^1, I) \to (\AC, \mu)$ in $\Bemb^1$ is $f
\mapsto f\mu_X$.  This map is an isomorphism (the Radon--Nikodym theorem),
but that does not seem to be an easy consequence of the universal property
of $(L^1, I)$.

Finally, consider the case $p = 2$. Theorem~\ref{thm:Bp} characterizes
$(L^2, I)$ as the initial object of $\Bpar^2$, but there is an alternative
characterization. Write $\Hilb$ for the category of Hilbert spaces and
linear contractions. Let $\cat{H}$ be the category of pairs $(F, v)$
consisting of a functor $F \from \Mpar^\op \to \Hilb$ and an element $v_X
\in F(X)$ for each measure space $X$, subject to axioms
\Kcomp--\Knorm\ (with $p = 2$ in \Knorm) and:
\begin{itemize}
\item[\KKHilb]
$\ip{(Fi)(v_Y)}{(Fj)(v_Z)} = 0$ whenever $Y \toby{i} X \otby{j} Z$ are
embeddings with disjoint images.
\end{itemize}
Thus, the difference between the categories $\Bpar^2$ and $\cat{H}$ is that
$\Ban$ has been replaced by $\Hilb$ and~\Kineq\ by~\KHilb.

The functor $L^2 \from \Mpar^\op \to \Hilb$, together with the constant
functions $I_X \in L^2(X)$, defines an object $(L^2, I)$ of
$\cat{H}$. It is universal as such:

\begin{propn}[Universal property of the $L^2$ functor]
\label{propn:Hilb}
$(L^2, I)$ is the initial object of $\cat{H}$.
\end{propn}

\begin{proof}
In the proof of Theorem~\ref{thm:Bp}, the only point where
axiom~\Kineq\ was used was to prove the inequality~\eqref{eq:BN-iter-ineq},
which in the case $p = 2$ states that
\begin{equation}
\label{eq:Hilb-proxy}
\biggl\| \sum_{c \in \F} c v^X_{f^{-1}(c)} \biggr\|^2
\leq
\sum_{c \in \F} 
\bigl\|c v^X_{f^{-1}(c)} \bigr\|^2
\end{equation}
for any simple function $f$ on a measure space $X$ and any object $(F, v) \in
\Bpar^2$. So, it suffices to prove that~\eqref{eq:Hilb-proxy} also holds
for any $(F, v) \in \cat{H}$. Indeed, recall that if we write $i_c$ for the
inclusion $f^{-1}(c) \incl X$ then by definition, $v^X_{f^{-1}(c)} =
(Fi_c)(v_{f^{-1}(c)})$. Axiom~\KHilb\ therefore implies that the elements
$v^X_{f^{-1}(c)}$ of $F(X)$ are orthogonal for distinct $c$, giving
equality in~\eqref{eq:Hilb-proxy}.
\end{proof}

\paragraph{Acknowledgements} I thank Mark Meckes for allowing me to include
Proposition~\ref{propn:mm}, and Ruijun Lin for helpful comments and
corrections. This work was supported at different times by an EPSRC
Advanced Research Fellowship and a Leverhulme Trust Research Fellowship.

\bibliography{mathrefs}

\end{document}

%% file: ab.tex
\usepackage{latexsym}
\usepackage{amssymb,amsmath}
\usepackage{mathrsfs}
\usepackage{stmaryrd}
\newcommand{\blank}{(\hspace*{1ex})}
\newcommand{\dashbk}{-}
\newcommand{\bref}[1]{(\ref{#1})}
\newcommand{\cat}[1]{\mathscr{#1}}
\newcommand{\fcat}[1]{\mathbf{#1}}
\newcommand{\ovln}[1]{\overline{#1}}
\newcommand{\such}{:}

\newcommand{\op}{\mathrm{op}}
\newcommand{\id}{\mathrm{id}}
\newcommand{\HOM}{\fcat{Hom}}
\newcommand{\R}{\mathbb{R}}

\newcommand{\C}{\mathbb{C}}
\newcommand{\Z}{\mathbb{Z}}
\newcommand{\demph}[1]{\textbf{#1}}
\newcommand{\iso}{\cong}
\newcommand{\nat}{\mathbb{N}}	
\newcommand{\of}{\mathbin{\circ}}
\newcommand{\sub}{\subseteq}
\newcommand{\cell}[4]{\put(#1,#2){\makebox(0,0)[#3]{#4}}}
\DeclareMathOperator{\colim}{colim}
\newcommand{\ot}{\leftarrow}
\newcommand{\toby}[1]{\stackrel{#1}{\longrightarrow}}
\newcommand{\otby}[1]{\stackrel{#1}{\longleftarrow}}
\newcommand{\ltoby}[1]{\xrightarrow{#1}}
\newcommand{\Vect}{\fcat{Vect}}
\newcommand{\ip}[2]{\langle #1, #2 \rangle}
\newcommand{\dee}{\,d}
\newcommand{\dx}{\dee x}
\newcommand{\incl}{\hookrightarrow}
\newcommand{\from}{\colon}
\newcommand{\hlf}{{\textstyle \frac{1}{2}}}
\newcommand{\F}{\mathbb{F}}
\newcommand{\Hilb}{\fcat{Hilb}}
\newcommand{\Ban}{\fcat{Ban}}
\newcommand{\NVS}{\fcat{NVS}}
\newcommand{\Seqcat}{\cat{D}}
\newcommand{\BVemb}{\cat{V}_{\textup{emb}}}
\newcommand{\BNemb}{\cat{N}_{\textup{emb}}}
\newcommand{\Bemb}{\cat{B}_{\textup{emb}}}
\newcommand{\BVpar}{\cat{V}}
\newcommand{\BNpar}{\cat{N}}
\newcommand{\Bpar}{\cat{B}}
\newcommand{\BVmap}{\psi}
\newcommand{\BNmap}{\phi}

\newcommand{\Measword}{Mea}
\newcommand{\Memb}{\fcat{\Measword}_{\textup{emb}}}
\newcommand{\Mpar}{\fcat{\Measword}}

\newcommand{\Mpres}{\fcat{\Measword}_{\textup{pres}}}

\newcommand{\Simp}{\mathscr{S}}
\newcommand{\simp}{S}

\newcommand{\dyfn}{E}
\newcommand{\AC}{AC}
\newcommand{\dsl}[1]{\mathbin{\boxplus_{#1}}}
\newcommand{\Kstyle}[1]{#1}
\newcommand{\KKstyle}[1]{\textbf{#1}}
\newcommand{\Kcomp}{\Kstyle{(I)}}
\newcommand{\KKcomp}{\KKstyle{(I)}}
\newcommand{\Kpres}{\Kstyle{(II)}}
\newcommand{\KKpres}{\KKstyle{(II)}}
\newcommand{\Knorm}{\Kstyle{(III)}}
\newcommand{\KKnorm}{\KKstyle{(III)}}
\newcommand{\Kineq}{\Kstyle{(IV)}}
\newcommand{\KKineq}{\KKstyle{(IV)}}
\newcommand{\KHilb}{\Kstyle{(IV$_\textrm{H}$)}}
\newcommand{\KKHilb}{\KKstyle{(IV$_\textrm{H}$)}}

%% file: thmlist.tex
\usepackage[amsmath,thmmarks]{ntheorem}

\newtheorem{thm}{Theorem}[section]
\newtheorem{propn}[thm]{Proposition}
\newtheorem{lemma}[thm]{Lemma}
\newtheorem{cor}[thm]{Corollary}

\newtheorem*{thma}{Theorem~A}
\newtheorem*{thmb}{Theorem~B}

\theorembodyfont{\normalfont}

\newtheorem{remark}[thm]{Remark}
\newtheorem{remarks}[thm]{Remarks}

\theoremstyle{nonumberplain}
\theoremsymbol{\ensuremath{\square}}
\qedsymbol{\ensuremath{\square}}

\newtheorem{proof}{Proof}

\newcommand{\theoremtobeproved}{}
\newtheorem{pfoftheorem}{Proof of \theoremtobeproved}